\documentclass[12pt,reqno]{amsart}
\usepackage{cmap,mathtools}
\usepackage[T1]{fontenc}
\usepackage[utf8]{inputenc}
\usepackage[english]{babel}
\usepackage{amsfonts,amssymb,amsmath}
\usepackage{footnote,bigints}
\usepackage{array}
\usepackage{adjustbox}
\usepackage{graphicx}
\usepackage{caption}
\usepackage{subcaption}
 \usepackage{enumitem}

\usepackage[pagewise]{lineno}

\DeclareUnicodeCharacter{0308}{\"{}}

\allowdisplaybreaks
\usepackage{graphicx}
\usepackage{epstopdf}
\usepackage{a4wide}
\usepackage{xcolor}
\usepackage[colorlinks=true,linktocpage,pdfpagelabels,
bookmarksnumbered,bookmarksopen]{hyperref}
\definecolor{ForestGreen}{rgb}{0.1,0.6,0.05}
\definecolor{EgyptBlue}{rgb}{0.063,0.1,0.6}
\hypersetup{
	colorlinks=true,
	linkcolor=EgyptBlue,         
	citecolor=ForestGreen,
	urlcolor=olive
}

\usepackage[hyperpageref]{backref}

\newtheorem{theorem}{Theorem}[section]
\newtheorem{proposition}{Proposition}[section]
\newtheorem{definition}{Definition}[section]
\newtheorem{lemma}{Lemma}[section]
\newtheorem{remark}{Remark}[section]
\newtheorem{cor}{Corollary}[section]

\numberwithin{equation}{section}
\numberwithin{theorem}{section}
\numberwithin{equation}{section}
\numberwithin{theorem}{section}

\usepackage{ulem}
\usepackage{xcolor}
\usepackage[colorlinks=true,linktocpage,pdfpagelabels,
bookmarksnumbered,bookmarksopen]{hyperref}
\definecolor{ForestGreen}{rgb}{0.1,0.6,0.05}
\definecolor{EgyptBlue}{rgb}{0.063,0.1,0.6}
\hypersetup{
	colorlinks=true,
	linkcolor=EgyptBlue,         
	citecolor=ForestGreen,
	urlcolor=olive
}

\subjclass[2010]{Primary  58J50; Secondary 35P15}
\usepackage[hyperpageref]{backref}
\usepackage[foot]{amsaddr}

\DeclareUnicodeCharacter{2212}{-}

\title [Mixed Steklov Neumann eigenvalues]{ Sharp bounds and geometric properties of the first non trivial Steklov Neumann Eigenvalue  }

\keywords{Steklov--Neumann Eigenvalues, Doubly connected domains, Nodal domains, Symmetries, Star Shaped domain}
\author{Sagar Basak, Gloria Paoli, Rossano Sannipoli, Sheela Verma}

\newcommand{\Addresses}{
  \bigskip 
   \textit{E-mail address}, S.~Basak: \texttt{sagarbasak.rs.mat22@itbhu.ac.in}
   \medskip
\textsc{ Department of Mathematical Sciences, Indian Institute of Technology (BHU), Varanasi, India}
 
  \textit{E-mail address}, G.~Paoli: \texttt{gloria.paoli@unina.it} 
 \medskip
  \textsc{Dipartimento di Matematica e Applicazioni ``R. Caccioppoli'', Universit\`a degli studi di Napoli Federico II, Via Cintia, Complesso Universitario Monte S. Angelo, 80126 Napoli, Italy.}

  \textit{E-mail address}, R.~Sannipoli: \texttt{rossano.sannipoli@fjfi.cvut.cz} 
     \medskip 
\textsc{Department of Mathematics, Faculty of Nuclear Sciences and Physical Engineering, Czech Technical University in Prague, Trojanova 13, 120 00, Prague, Czech Republic.}

\textit{E-mail address}, S.~Verma: \texttt{sheela.mat@iitbhu.ac.in}
 \medskip
 \textsc{ Department of Mathematical Sciences, Indian Institute of Technology (BHU), Varanasi, India}\par\nopagebreak 

}

\begin{document}

\begin{abstract}
In this article, we study the mixed Steklov--Neumann eigenvalue problem on doubly connected domains. First, we show that among all doubly connected domains in $\mathbb{R}^n$ of the form $B_{R_2}\setminus \overline{B_{R_1}}$, where $B_{R_1}$ and $B_{R_2}$ are open balls of fixed radii satisfying $\overline{B_{R_1}} \subset B_{R_2}$, the first non-zero Steklov--Neumann eigenvalue attains its maximal value when the balls are concentric. Next, we establish bounds for the first non-zero Steklov--Neumann eigenvalue on a doubly connected star-shaped domain contained in a hypersurface equipped with a revolution-type metric. We also derive the asymptotic behavior of the first non-zero Steklov--Neumann eigenvalue on a bounded domain with a spherical hole in $\mathbb{R}^n$ as the radius of the hole approaches zero. Finally, we study the number of nodal domains of the eigenfunction corresponding to the first non zero Steklov--Neumann eigenvalue on a bounded domain in $\mathbb{R}^n$ having a spherical hole. 
\\ \\
\noindent\textsc{\textbf{MSC 2020}:} 35B40, 35J25, 35P15. \\
\textsc{\textbf{Keywords}}:  Laplacian eigenvalue, Steklov--Neumann boundary conditions, Isoperimetric inequalities, Star shaped domains.
\end{abstract}
\maketitle

\section{Introduction}


 Let $(M, g)$ be a Riemannian manifold of
dimension $n$, and let $\Omega \subset M$ be a bounded doubly connected
domain with Lipschitz boundary
$\partial \Omega = \Gamma_1 \cup \Gamma_2$, where $\Gamma_1$ and
$\Gamma_2$ are disjoint. We consider the following Steklov-Neumann eigenvalue problem for the Laplacian
\begin{equation}\label{SN Problem}
\begin{cases}
\Delta u = 0 & \text{in } \Omega, \\
\dfrac{\partial u}{\partial \nu} = 0 & \text{on } \Gamma_1, \\
\dfrac{\partial u}{\partial \nu} = \mu u & \text{on } \Gamma_2,
\end{cases}
\end{equation}
where $ \nu$ denotes the unit outward normal to the boundary $\partial\Omega$.

The weak formulation of \eqref{SN Problem} is given by
\begin{equation}\label{weak formulation SN}
\int_{\Omega}  \nabla u \nabla v \, dV
= \mu \int_{\Gamma_2} u v \, dS,
\quad \text{for all } v \in H^1(\Omega).
\end{equation}

An equivalent spectral interpretation can be formulated in terms of the
Dirichlet-to-Neumann operator on $\Gamma_2$. For $F \in L^2(\Gamma_2)$,
let $\widetilde{F}$ denote its harmonic extension to $\Omega$ satisfying
the Neumann condition on $\Gamma_1$. The Dirichlet-to-Neumann map, that is 
\[
D : L^2(\Gamma_2) \to L^2(\Gamma_2),
\quad F \mapsto \frac{\partial \widetilde{F}}{\partial \nu},
\]
is compact, positive, and self-adjoint
(see \cite[Chapter~III]{bandle1980isoperimetric}). Consequently, its
spectrum is discrete and can be arranged as
\[
0 = \mu_0(\Omega) < \mu_1(\Omega) \leq \mu_2(\Omega)
\leq \cdots \nearrow \infty.
\]
The eigenfunction corresponding to $\mu_0(\Omega)$ is constant. The first
non-zero eigenvalue $\mu_1(\Omega)$ admits the following variational characterization
\begin{equation}\label{characterization}
\mu_1(\Omega) =
\inf \left\{
\frac{\displaystyle \int_{\Omega} |\nabla u|^2 \, dV}
{\displaystyle \int_{\Gamma_2} u^2 \, dS}
:\;
u \in H^1(\Omega), \quad
\int_{\Gamma_2} u \, dS = 0
\right\}.
\end{equation}

\subsection{State of the Art}

Spectral problems on doubly connected domains have been widely studied under
various boundary conditions. One of the earliest contributions is due to P{\'o}lya~\cite{polya1960two}, who derived upper bounds for the first Dirichlet eigenvalue of planar ring-shaped domains. In 
\cite{kesavan2003two} it is proved that, among annular domains in $\mathbb{R}^n$
obtained by translating the inner ball, the first Dirichlet eigenvalue is maximized when the balls are concentric. This optimality property was later
extended to space forms in ~\cite{anisa2005two}. For mixed Dirichlet--Neumann boundary conditions, Payne and Weinberger
\cite{payne1961some} established an analogous extremal result. The authors showed
that the concentric annulus maximizes the first Dirichlet--Neumann eigenvalue among multiply connected planar domains with prescribed geometric
constraints.

Mixed Steklov-type problems have also attracted considerable attention.
In ~\cite{verma2020eigenvalue}, the authors proved that concentric
annuli maximize the first eigenvalue of the Steklov--Dirichlet problem among
domains of fixed volume. An alternative proof and further extensions,
including the classical Steklov problem, were later provided
in~\cite{ftouhi2022place}. Bounds and isoperimetric inequalities for higher
Steklov--Dirichlet eigenvalues on perforated domains were recently obtained
in~\cite{Basak_Chorwadwala_Verma_2025, sannipoli2025estimates}. Eigenvalue estimates for doubly connected star-shaped domains have been
studied in various geometric settings. In~\cite{gavitone2023isoperimetric},
upper bounds for the first Steklov--Dirichlet eigenvalue were derived,
showing that the corresponding concentric annulus provides an extremal
configuration. These results were later extended to non-Euclidean space
forms in~\cite{Basak_Chorwadwala_Verma_2025}. Using related techniques,
lower bounds for the first Steklov--Robin eigenvalue were established
in~\cite{gavitone2023steklov}. The asymptotic behavior of Steklov-type eigenvalues on perforated domains has
also been investigated:  in~\cite{sannipoli2025estimates}, it was shown that,
as the radius of a spherical hole tends to zero, the first
Steklov--Dirichlet eigenvalue converges to zero, while the second eigenvalue
converges to the first non-trivial Steklov eigenvalue of the limiting
domain.
Other reference for mixed
Steklov--Dirichlet problems are the following 
\cite{ gavitone2024monotonicity,
hong2020shape, GloriaPaoli2021C, verma2020eigenvalue}.

For eigenfunctions corresponding to the Steklov eigenvalues on a bounded simply connected domain in $\mathbb{R}^2$, the authors in \cite{KS1969} studied the structure of their nodal sets. The authors proved that the nodal set of an eigenfunction associated with the $k$-th Steklov eigenvalue on a bounded simply connected domain $\Omega \subset \mathbb{R}^2$ divides $\Omega$ into at most $k$ subdomains, and moreover, that no nodal line forms a closed curve.
In \cite{AM1994}, the authors determined the number of nodal domains of an eigenfunction corresponding to the $k$-th Steklov eigenvalue on a simply connected bounded domain in $\mathbb{R}^n$. More recently, in \cite{sannipoli2025estimates}, the author generalized this result to the mixed Steklov--Dirichlet eigenvalue problem on domains with a hole in $\mathbb{R}^n$.

 Other types of mixed boundary conditions have been investigated in recent years: for instance, mixed Dirichlet--Neumann
problem \cite{anoop2022szego1, anoop2020reverse, anoop2021shape}, and the mixed
Robin--Neumann problem
\cite{cito2025optimality, cito2025stability, paoli2020sharp}. These studies highlight the delicate dependence of eigenvalues on the geometry of the
domain and reveal the presence of extremal configurations analogous to those
of classical Dirichlet or Steklov problems.

The mixed Steklov--Neumann has been only partially studied in 
\cite{MR2758960, basak2024bounds, hassannezhad2020eigenvalue,
kao2025extremal} and still presents several fundamental open questions. In particular, the ones concerning optimal shapes, sharp bounds, and asymptotic behavior. The aim of the present paper is to address these problems.  
\subsection{Main Results}

We now describe the main contributions of this work. Our first result concerns
a shape optimization problem for particular doubly connected domains in the Euclidean setting.
More precisely, we determine the configuration that maximizes the first
non-zero Steklov--Neumann eigenvalue among eccentric annuli with prescribed
radii. This provides the Steklov--Neumann counterpart of classical extremal
results known for Dirichlet, Dirichlet--Neumann, and Steklov--Dirichlet
problems. To state the first results, we need to introduce the following notation. Let $R>0$,  we define $B_{R}(d) = y_d + B_{R}$, where $y_d = (0, 0, \dots, d)$, the ball centered at the point $y_d$ of radius $R>0$. If $d=0$, we will denote the ball by $B_R$.
\begin{theorem} \label{theorem: optimal}
Let $R_2>R_1>0$ and let  $\Omega_d = B_{R_2}(d) \setminus \overline{B_{R_1}},$
 where $B_{R_1}$ and $B_{R_2}(d)$ are such that $\overline{B_{R_1}} \subset B_{R_2}(d)$. Then, 
\begin{equation}
     \mu_1(\Omega_d)\leq \mu_1(\Omega_0).
 \end{equation}
 
\end{theorem}
\noindent Theorem~\ref{theorem: optimal} shows that the concentric annulus maximizes the first non-trivial Steklov-Neumann eigenvalue in the class of all eccentric annuli. Despite the different spectral
nature of the problem, the extremal configuration coincides with that
observed for other boundary conditions, showing that the symmetry plays a decisive role
also in the Steklov--Neumann framework.

Our second result addresses eigenvalue estimates on more general geometric
settings. Motivated by recent advances on star-shaped domains for mixed
Steklov-type problems, see \cite{basak2024bounds}, we derive explicit lower and upper bounds for the
first Steklov--Neumann eigenvalue on doubly connected star-shaped domains
embedded in hypersurfaces in the Riemannian manifold $M$, endowed with metrics of revolution. 
Let $p\in M$ and  $\tilde{\Omega}_{\mathrm{out}} \subset M$ be a star-shaped domain with respect to $p$. Let 
$B_{R_1}$ denote the geodesic ball in $M$ of radius $R_1$ centered at $p$, such that $B_{R_1} \subset \tilde{\Omega}_{\mathrm{out}}$, and let us
define $\tilde{\Omega} = \tilde{\Omega}_{\mathrm{out}} \setminus \overline{B_{R_1}}.$
Let $B_{R_M}$ be the smallest geodesic ball in $M$ centered at $p$ with radius $R_M$ that contains 
$\tilde{\Omega}_{\mathrm{out}}$, and let $B_{R_m}$ be the largest geodesic ball in $M$ centered at $p$ with radius 
$R_m$ that is contained in $\tilde{\Omega}_{\mathrm{out}}$. With this setup, we obtain the following result.

\begin{theorem}\label{thm:upper-lower}
Let $M$ be a hypersurface of revolution  with  metric $g$ and let $\tilde{\Omega}$, $B_{R_1}$, $B_{R_M}$, and $B_{R_m}$ be defined as above. Then there exist two explicit and postive constants $C_i=C_i(g,R_m,R_M)$, $i=1,2$, such that 
\begin{align*}
C_1\;
\mu_1\!\left(B_{R_m}\setminus \overline{B_{R_1}}\right)
\;\le\; \mu_1(\tilde{\Omega}) 
\;\le\;
C_2\;
\mu_1\!\left(B_{R_M}\setminus \overline{B_{R_1}}\right).
\end{align*}
\end{theorem}


These bounds reveal that the geometry of star-shaped domains can be
quantitatively controlled through suitable comparison with concentric
annuli. In particular, they extend known inequalities for
Steklov--Dirichlet and Steklov--Robin problems to the
Steklov--Neumann setting.

 Our third result regards the asymptotic behavior of the spectrum when the
inner hole shrinks. Understanding this limit is crucial both from the
analytic and geometric viewpoint, as it clarifies the relation between the
Steklov--Neumann problem on perforated domains and the classical Steklov
problem on simply connected domains.

\begin{theorem}\label{thm:convergence}
Let $\Omega_{\mathrm{out}} \subset \mathbb{R}^n$ be a bounded, connected domain with Lipschitz boundary, and let $B_r$ denote the ball of radius $r>0$ centered at the origin such that $\overline{B_r} \subset \Omega_{\mathrm{out}}$.  
Define the punctured domain
$\Omega_r = \Omega_{\mathrm{out}} \setminus \overline{B_r}.$
Then we have 
$$
\lim_{r\to 0}\mu_1(\Omega_r) = \sigma_1(\Omega_{\mathrm{out}}),
$$
where $\sigma_1(\Omega_{\mathrm{out}})$ is the first non-trivial Steklov eigenvalue. Moreover, there exists a sequence of eigenfunctions $\{u_1^r\}_r$ corresponding to $\mu_1(\Omega_r)$ and an eigenfunction $\overline{u_1}$ corresponding to $\sigma_1(\Omega_{\mathrm{out}}),$ such that $u_1^r$ converges to $\overline{u_1}$ in $H^1(\Omega_{\mathrm{out}}),$ as $r\to 0.$ 
\end{theorem}

The asymptotic convergence described in
Theorem~\ref{thm:convergence} has significant geometric consequences.
In particular, it allows us to derive isoperimetric-type inequalities for
the first nonzero Steklov--Neumann eigenvalue on domains with sufficiently small
holes. More precisely, for $r$ small enough, the eigenvalue
$\mu_1(\Omega_r)$ can be controlled by the corresponding eigenvalue of a
concentric annulus satisfying either a measure or a perimeter constraint,
as established in Corollaries~\ref{thm:isoperimetric1}
and~\ref{thm:isoperimetric2}.

These estimates should be compared with the classical inequalities
available for non-perforated domains. In the Steklov setting, in \cite{Bro2001} a sharp upper bound is provided for the first non-trivial Steklov
eigenvalue under a volume constraint, while the Weinstock inequality
characterizes the ball as the maximizer under a perimeter constraint
(in dimension two for simply connected domains, see \cite{Wein1954}, and in higher dimensions
within the class of convex sets, \cite{BFNT2021}). The results obtained here show that,
in the vanishing-hole regime, analogous comparison principles continue to
hold for the mixed Steklov--Neumann problem.

Finally, we investigate the number of nodal domains of the eigenfunctions corresponding to the first non-trivial Steklov-Neumann eigenvalue on a doubly connected domain. A nodal domain of a function is a connected component contained in $\Omega$ where the function does not change sign (for the precise definition see Section \ref{sec:6}).

\begin{theorem}
\label{prop:nodaldomains}
    Let $\Omega$ be a bounded doubly connected domain with Lipschitz boundary. The nodal domains of the eigenfunctions corresponding to $\mu_1(\Omega)$ are exactly $2$.
    
\end{theorem}

\subsection{Organization of the Paper}

The paper is organized as follows. Section~2 contains preliminary results. In Section~3 we prove Theorem \ref{theorem: optimal},  establishing the optimality of concentric annuli.
In Section~4 we prove Theorem \ref{thm:upper-lower}, deriving eigenvalue bounds for star-shaped domains. Section~5 is dedicated to the proof of Theorem \ref{thm:convergence}, addressing 
 the asymptotic analysis of the mixed eigenvalue as the hole shrinks. Section~6 is devoted to the study of nodal domain
properties, proving Theorem \ref{prop:nodaldomains}.

\section{Preliminary results}
\subsection{Standard notations and some useful parametrizations in the Riemannian setting}
Let $M$ be a hypersurface of revolution  with  metric 
   $g = dr^2 + h^2(r)\, g_{\mathbb{S}^{n-1}},$
where $g_{\mathbb{S}^{n-1}}$ denotes the standard metric on the unit sphere 
$\mathbb{S}^{n-1}$, and $r \in [0, L]$ for some $L > 0$. 
Further, we assume that the function $h:[0,L]\to\mathbb{R}$ satisfies $h(0) = 0,  h'(0) = 1$, and $h(r)$ is an increasing function of $r.$
Let us fix a point $p\in  M$ and consider  $\tilde{\Omega}_{out} \subset M$  a bounded, smooth, and star-shaped domain with respect to  $p$. Let us introduce the following notation: we denote by $T_pM$ the tangent space of $M$  at the point $p\in M$ and by $exp_p$ the exponential map defined from $T_pM$ to $M$.
Let $\|u\| := \sqrt{\langle u, u\rangle}_g$. For each $p \in M$, define the unit sphere in the tangent space at $p$ by
\begin{align*}
    U_pM:=\{u\in T_pM: \|u\|=1\}.
\end{align*}
Note that corresponding to each $u \in U_pM,$ there exist a unique point $q\in \partial \tilde{\Omega}_{out}$ such that $q=exp_p(r_uu)$ for some $r_u>0.$ Then the domain $\tilde{\Omega}_{out}$ and its boundary $\partial\tilde{\Omega}_{out}$ can be expressed as 
 \begin{align*}
     \tilde{\Omega}_{out}=\{exp_p(tu): u\in U_pM, 0<t<r_u\}, \quad \partial \tilde{\Omega}_{out}= \{exp_p(r_uu):u\in U_pM\}.
 \end{align*}
We denote by
\begin{equation}\label{inradius}
    \displaystyle R_{m} := \min_{u \in U_{p}
M}r_u, \quad\displaystyle R_{M} := \max_{u \in U_{p}M} r_u .
\end{equation}

Let $\partial_r$ denote the radial vector field starting from $p$, and let $\nu $ be the unit outward normal to $\partial \tilde{\Omega}_{out}.$ For any point $q\in \partial \tilde{\Omega}_{out}$, let $\theta(q)$ denote the angle between $\nu (q)$ and $\partial_r(q)$. Since $\tilde{\Omega}_{out}$ is a bounded smooth star-shaped domain,
$$ cos(\theta(q))=\langle \nu (q), \partial_r(q)\rangle >0, $$
that implies 
$\theta(q)<\frac{\pi}{2}$. By compactness of $\partial\tilde{\Omega}_{out},$ there exist a constant $\alpha$ such that $0\leq \theta(q)\leq \alpha < \frac{\pi}{2}$ for all $q\in \partial \tilde{\Omega}_{out}.$  Let $a = \tan^{2} \alpha$.
Thus, for all $q \in \partial \tilde{\Omega}_{out}$,  
\begin{align} \label{eqn: constt a}
  \frac{\|\overline{\nabla} R_u \|^2}{h^2(R_u)} 
= \tan^{2}\!\big(\theta(q)\big) 
\;\leq\; \tan^{2}(\alpha) = a,  
\end{align}
where $\overline{\nabla} R_u$ denotes the component of $\nabla R_u$ tangent to $\mathbb{S}^{n-1}$. For more details, see Section~2 of~\cite{Basak_Chorwadwala_Verma_2025} and the references therein.

\subsection{The radial case}\label{subsec:radial}
 Let $R>0$,  we define $B_{R}(d) = y_d + B_{R}$, where $y_d = (0, 0, \dots, d)$, the ball centered at the point $y_d$ of radius $R>0$. If $d=0$, we will denote the ball by $B_R$. Let us fix $R_2>R_1>0$ and for $d \in [0, R_2 - R_1)$, we denote the set $\Omega_d := B_{R_2}(d) \setminus \overline{B_{R_1}}$. Let us consider the following Steklov--Neumann eigenvalue problem on  $\Omega_d$ 
\begin{align*}
        \begin{cases}
            \Delta u =0 \quad \text{in} \,\,\,\Omega_d=B_{R_2}(d) \setminus \overline{B_{R_1}}\\
            \frac{\partial u}{\partial \nu}=0 \quad \text{on}\,\,\, B_{R_1}\\
            \frac{\partial u}{\partial \nu}= \mu u \quad \text{on} \,\,\, B_{R_2}(d),
        \end{cases}
\end{align*}
where $\nu$ denote the outward unit normal vector to the boundary. For $d = 0$, we have $\Omega_0 = B_{R_2} \setminus \overline{B_{R_1}}$. In this case, $\mu_l(\Omega_0)$ can be explicitly computed and it is given by (see \cite{basak2024bounds})
 \begin{align*}
      \mu_l(\Omega_0) = \dfrac{l(l + n - 2)\left( \left( \frac{R_2}{R_1} \right)^{2l + n - 2} - 1 \right)}{R_2 \left( (l + n - 2)\left( \frac{R_2}{R_1} \right)^{2l + n - 2} + l \right)}, \quad  \quad l\geq 0.
 \end{align*}
 Moreover, the corresponding eigenfunctions are of the form
 $$w_l^j(t,\theta)= \bigg(r^l + \dfrac{l R_1^{2l + n - 2}}{(l + n - 2) r^{l + n - 2}}\bigg)Y_l^j(\theta),$$   where $Y_l^j(\theta)$ are the spherical harmonics, i.e. the eigenfunction of the Laplace Beltrami $\Delta_{\mathbb{S}^{n-1}}$ corresponding to the eigenvalue $l(l+n-2)$. The radial symmetry implies that  $\mu_l(\Omega_0)$ has the same multiplicity of the eigenvalue $l(l+n-2)$. In particular, the multiplicity of the first non-zero Steklov-Neumann eigenvalue is $n,$ and the corresponding eigenfunctions written in Cartesian coordinates are 
 \begin{equation*}
    u^i_n(x_1,x_2,\dots, x_n)= \left(r+\frac{R_1^n}{(n-1)r^{n-1}}\right)\frac{x_i}{r},
 \end{equation*}
where $i=1,2,\dots, n.$
For a fixed $i\in \{1,2,\dots,n-1\}, \,\, B_{R_2}(d)$ is symmetric with respect to the hyperplane $\{x_i=0\}$, so $$u_n^i(x_1,\dots,-x_i, \dots x_n)= -u_n^i(x_1,\dots,x_i, \dots x_n).$$ Using this condition we obtain  
\begin{align*}
    \int_{\partial B_{R_2}(d)} u_n^i ds &= \int_{\partial B_{R_2}(d) \cap \{x_i\geq 0\}} u_n^i ds \,\,+ \int_{\partial B_{R_2}(d) \{x_i\leq 0\}} u_n^i ds\\
    &= \int_{\partial B_{R_2}(d) \cap \{x_i\geq 0\}} u_n^i ds \,\,- \int_{\partial B_{R_2}(d) \{x_i\geq 0\}} u_n^i ds=0.
\end{align*}
 Therefore, for every  $i\in \{1,2, \dots, n-1\} $, the eigenfunction $u_n^i$ can be taken as a test function in the variational characterization of $\mu_1(\Omega_d).$

 \subsection{Some preliminary Lemmas}
Let us recall the spherical coordinates in $n$ dimensions, that are
\begin{align*}
    \begin{cases}
    x_1=r\sin \theta_1 \sin \theta_2 \dots \sin \theta_{n-2}\sin \theta_{n-1},\\
    x_2=r\sin \theta_1 \sin \theta_2 \dots \sin \theta_{n-2}\cos \theta_{n-1},\\ 
    \quad \quad \vdots\\
     x_{n-1}= r\sin \theta_1 \cos \theta_2,\\
     x_n= r\cos \theta_1,
\end{cases}
\end{align*}
where $(r, \theta_1, \dots, \theta_{n-1})\in \mathbb{R^+} \times [0,\pi]\times,\dots,\times [0,\pi]\times [0,2\pi].$
Let $P \in \partial B_{R_2}(d)$, and denote by $R_d(\theta_1)$ the distance between the origin $O$ and the point $P$.  Here, $\theta_1$ represents the angle formed by the line segment $OP$ with the $x_1$–axis. Then, 
\begin{align*}
   R_d(\theta_1) = d \cos \theta_1 + \sqrt{R_2^2 - d^2 \sin^2 \theta_1}, 
\end{align*}
and 
\begin{align*}
  \sqrt{R_d(\theta_1)^2 + {R_{d}'^2(\theta_1)}}
= R_2 \left( 1 + \frac{d \cos \theta_1}{\sqrt{R_2^2 - d^2 \sin^2 \theta_1}} \right).  
\end{align*}
For more details, we refer the reader to~\cite{ftouhi2022place}.

Since for every $i\in \{1,2, \dots, n-1\},$ eigenfunction $u_n^i$ can be used as a test function in the variational characterization of $\mu_1(\Omega_d)$, we take $f=u_n^{n-1}$. Using spherical coordinates, we have 
 \begin{align} \label{test function}
     f(r,\theta)= \begin{cases}
         \sin \theta_1 \left(r+\frac{R_1^2}{r}\right) \quad \text{for} \quad n=2,\\
         \sin \theta_1 \cos \theta_2 \left(r+\frac{R_1^2}{r}\right) \quad \text{for} \quad n \geq 3.
     \end{cases}
     \end{align}
Next, we recall the following lemmas, which will be used in the proof of Theorem \ref{theorem: optimal}. A detailed proof of these two lemmas are provided in~\cite{ftouhi2022place}.
\begin{lemma} \label{A comparison}
For every $n \geq 2$ and $d \in [0, R_2 - R_1)$, define
\[
\begin{cases}
A_1(d) = \displaystyle \int_{0}^{\pi} \sin^{n-2}\theta_1 \, \big( R_d^n(\theta_1) - R_1^n \big) \, d\theta_1, \\[8pt]
A_2(d) = \displaystyle \int_{0}^{\pi} \phi(\theta_1)
\ln\!\left( \frac{R_d(\theta_1)}{R_1} \right) \, d\theta_1, 
\quad \text{where} \quad 
\phi(\theta_1) = -n\sin^n\theta_1 + (n-1)\sin^{n-2}\theta_1, \\[8pt]
A_3(d) = \displaystyle \int_{0}^{\pi} \psi(\theta_1)
\left( \frac{1}{R_d^n(\theta_1)} - \frac{1}{R_1^n} \right) \, d\theta_1, 
\,\, \text{where}\,\,
\psi(\theta_1) = (n-2)\sin^n\theta_1 + (n-1)\sin^{n-2}\theta_1.
\end{cases}
\]
 Then the following holds:
\begin{enumerate}
    \item $A_1(d) = A_1(0)$,
    \item $A_2(d) = 0 $,
    \item $A_3(d) \geq A_3(0)$, with equality if and only if $d = 0$.
\end{enumerate}
\end{lemma}
\begin{lemma}\label{V comparison}
    For every $n \geq 2$ and $d \in [0, R_2 - R_1)$,  define 
    \begin{align*}
        \begin{cases}
            V_1(d)&=\displaystyle\int _{0}^{\pi} \sin^n\theta_1 R_d^n(\theta_1) \sqrt{R_d(\theta_1)^2 + {R_{d}'^2(\theta_1)}} \,\, d\theta_1\\
            V_2(d) &= \displaystyle\int _{0}^{\pi} \sin^n\theta_1 \frac{d\cos\theta_1}{\sqrt{R_2^2-d^2\sin^2\theta_1}} d\theta_1\\
            V_3(d)&=\displaystyle\int_0^{\pi} \frac{R_2 \sin^n\theta_1}{R_d(\theta_1)^{n-1} \sqrt{R_2^2-d^2\sin^2\theta_1}}d\theta_1.
        \end{cases}
    \end{align*}
    Then the following holds: 
    \begin{enumerate}
    \item $V_1(d) = V_1(0)$,
    \item $V_2(d) = 0$,
    \item $V_3(d) \geq V_3(0)$, with equality if and only if $d =0$.
\end{enumerate}
\end{lemma}
\noindent

\section{Optimality of Concentric Annuli for the First Steklov--Neumann Eigenvalue}

 In this section, we give the proof of Theorem \ref{theorem: optimal}.

\begin{proof} [Proof of Theorem \ref{theorem: optimal}]
Let $f \in H^1(\mathbb{R}^n \setminus \overline{B_{R_1}})$ be the  admissible test function defined in ~\eqref{test function} for  $\mu_1(\Omega_d)$. Then, using the variational characterization, we have
\begin{align*}
     \mu_1(\Omega_d) \leq 
    \frac{\displaystyle \int_{\Omega_d} |\nabla f|^2 \, dx}
         {\displaystyle \int_{\partial B_{R_2}(d)} f^2 \, dS}.
\end{align*}
To prove the Theorem, it suffices to prove the following two inequalities:
\begin{align}
    \int_{\Omega_{d}} |\nabla f|^2 \, dV 
    &\leq 
    \int_{\Omega_{0}} |\nabla f|^2 \, dV, 
    \label{numerator inequality formal}
    \\[4pt]
    \int_{\partial B_{R_2}(d)} f^2 \, dS 
    &\geq 
    \int_{\partial B_{R_2}} f^2 \, dS.
    \label{denominator inequality formal}
\end{align}

\medskip
\noindent
We first prove ~\eqref{numerator inequality formal}.  
For $n = 2$, the gradient of $f$ in polar coordinates $(r, \theta_1)$ is given by
\[
\nabla f(r, \theta_1)
= 
\begin{bmatrix}
\sin \theta_1 \left( 1 - \dfrac{R_1^2}{r^2} \right) \\[4pt]
\cos \theta_1 \left( 1 + \dfrac{R_1^2}{(n-1)r^2} \right)
\end{bmatrix}.
\]
Hence,

{\small
\begin{align*}
    \int_{\Omega_d}& |\nabla f|^2 \, dV=\\
    &= 
    \int_{0}^{2\pi}\!\!\int_{R_1}^{R_d(\theta_1)} 
    \Big[
        \sin^2 \theta_1 \!\left( 1 - \tfrac{R_1^2}{(n-1)r^2} \right)^{\!2}
        + 
        \cos^2 \theta_1 \!\left( 1 - \tfrac{R_1^2}{(n-1)r^2} \right)^{\!2}
    \Big]
    r \, dr \, d\theta_1 \\
    &=
    \int_{0}^{2\pi}
    \Bigg[
        \frac{R_d^2(\theta_1) - R_1^2}{2}
        + 
        \frac{2R_1^2}{n-1} (1 - 2\sin^2 \theta_1)
        \ln\!\left( \frac{R_d(\theta_1)}{R_1} \right) 
        - 
        \frac{R_1^4}{2(n-1)} 
        \!\left( \frac{1}{R_d^2(\theta_1)} - \frac{1}{R_1^2} \right)
    \Bigg]
    d\theta_1.
\end{align*}}

\medskip
For $n \geq 3$, we have
\[
\nabla f(r, \theta_1, \dots, \theta_{n-1})
= \begin{bmatrix}
\displaystyle \frac{\partial f}{\partial r} \\[4pt]
\displaystyle \frac{1}{r}\frac{\partial f}{\partial \theta_1} \\[4pt]
\displaystyle \frac{1}{r \sin \theta_1}\frac{\partial f}{\partial \theta_2} \\[4pt]
\vdots \\[4pt]
\displaystyle \frac{1}{r \sin \theta_1 \cdots \sin \theta_{n-2}}\frac{\partial f}{\partial \theta_{n-1}}
\end{bmatrix}
=
\begin{bmatrix}
\sin \theta_1 \cos \theta_2 \left( 1 - \dfrac{R_1^n}{r^n} \right) \\[4pt]
\cos \theta_1 \cos \theta_2 \left( 1 + \dfrac{R_1^n}{(n-1)r^n} \right) \\[4pt]
- \sin \theta_2 \left( 1 + \dfrac{R_1^n}{(n-1)r^n} \right) \\[4pt]
0 \\[4pt]
\vdots \\[4pt]
0
\end{bmatrix}.
\]
Then, following a similar calculation that the reader can find in~\cite{ftouhi2022place}, we obtain

\[
\begin{aligned}
\int_{\Omega_d} |\nabla f|^2 \, dV
&= 
\frac{2}{(n-1)^2} 
\prod_{k=0}^{n-3} I_k
\int_{0}^{\pi}
\Bigg[
\sin^{n-2}\!\theta_1 \,
\frac{(n-1)^2}{n}
\big( R_d^n(\theta_1) - R_1^n \big)
\\[4pt]
&\quad
+ 2R_1^n 
\ln\!\left( \frac{R_d(\theta_1)}{R_1} \right)
\!\big( -n\sin^n\!\theta_1 + (n-1)\sin^{n-2}\!\theta_1 \big)
\\[4pt]
&\quad
- \frac{R_1^{2n}}{n}
\!\left( \frac{1}{R_d^n(\theta_1)} - \frac{1}{R_1^n} \right)
\!\big( (n-2)\sin^n\!\theta_1 + (n-1)\sin^{n-2}\!\theta_1 \big)
\Bigg]
d\theta_1,
\end{aligned}
\]
where $I_k = \int_0^{\pi} \sin^k t \, dt$.  
Hence,
\[
    \int_{\Omega_d} |\nabla f|^2 \, dV
    =
    \frac{2}{(n-1)^2} \prod_{k=0}^{n-3} I_k
    \left[
        \frac{(n-1)^2}{n}A_1(d)
        + 2R_1^n A_2(d)
        - \frac{R_1^{2n}}{n}A_3(d)
    \right],
\]
where $A_1(d), A_2(d), A_3(d)$ are defined in~\eqref{A comparison}.  
Applying Lemma~\ref{A comparison}, we deduce
\[
    \int_{\Omega_{d}} |\nabla f|^2 \, dV 
    \leq 
    \int_{\Omega_{0}} |\nabla f|^2 \, dV.
\]

\medskip
\noindent
Then, to establish~\eqref{denominator inequality formal},
if $\theta= (\theta_1,...,\theta_{n-1})$ and $Q=[0,\pi]^{n-1}$, we know that
\begin{equation*}
\int_{\partial B_{R_2}(d)} f^2 \, dS
=2 \int_{Q}
f^2(r, \theta)
\, R_d(\theta_1)^{n-2}
\prod_{i=1}^{n-2} \sin^{n-1-i}\!\theta_i
\sqrt{R_d(\theta_1)^2 + \big(R_d'(\theta_1)\big)^2}
\, d\theta,
\end{equation*}
where $d\theta= d\theta_1\cdots d\theta_{n-1}$.
Using the explicit expression of $f$ and following a similar calculation as in~\cite{ftouhi2022place}, we obtain
\begin{align}
      \int_{\partial B_{R_2}(d)} f^2 \, dS 
    = 
    2 \prod_{k=2}^{n-1} 
    \left(
        V_1(d)
        + \frac{2R_1^n R_2}{n-1} 
        ( I_n + V_2(d) )
        + \frac{R_1^{2n}}{R_2(n-1)^2} V_3(d)
    \right),
\end{align}
where $V_1(d), V_2(d), V_3(d)$ are as defined previously.  
Then, by Lemma~\ref{V comparison}, we obtain
\begin{align*}
    \int_{\partial B_{R_2}(d)} f^2 \, dS 
    \geq 
    \int_{\partial B_{R_2}} f^2 \, dS.
\end{align*}

\medskip
\noindent
Finally, combining~\eqref{numerator inequality formal} and~\eqref{denominator inequality formal}, we get
\begin{align*}
    \mu_1(\Omega_{d})
    \leq 
    \frac{\displaystyle \int_{\Omega_{d}} |\nabla f|^2 \, dx}
         {\displaystyle \int_{\partial B_{R_2}(d)} f^2 \, dS}
    \leq
    \frac{\displaystyle \int_{\Omega_{0}} |\nabla f|^2 \, dx}
         {\displaystyle \int_{\partial B_{R_2}} f^2 \, dS}
    = 
    \mu_1(\Omega_{0}),
\end{align*}
with equality if and only if $d = 0$.  
This completes the proof.
\end{proof}

\section{Bounds for Steklov--Neumann eigenvalues}
Throughout this section, we consider $\tilde{\Omega} = \tilde{\Omega}_{\mathrm{out}} \setminus \overline{B_{R_1}}$ to be a star-shaped domain with a spherical hole, where $B_{R_1} \subset M$ is the geodesic ball of radius $R_1$ centered at a point $p$ such that $\overline{B_{R_1}} \subset\tilde{\Omega}_\mathrm{out}.$ We denote by $B_{R_m}$ and $B_{R_M}$ the geodesic balls in $M$, also centered at $p$, with radii $R_m$ and $R_M$, defined in \ref{inradius}, respectively.

\subsection{Useful results}
Let $w$ be a continuously differentiable function defined on the annular domain $B_{R_M}\setminus  \overline{B_{R_1}}$ in $M$. 
Then we have   
\begin{align*}
    \|\nabla w\|^2 \;=\; \left(\frac{\partial w}{\partial r}\right)^2 
\;+\; \frac{1}{h^2(r)} \,\|\overline{\nabla}w\|^2,
\end{align*}
where $\overline{\nabla}w$ denotes the component of $\nabla w$ tangent to $\mathbb{S}^{n-1}$.  

\begin{proposition} \label{upper bound inequality}
    Let $w$ be a continuously differentiable function defined on 
    $B_{R_M}\setminus \overline{B_{R_1}}$. Then the two following inequalities hold:
    
        \begin{align} \label{first}
        \int_{\tilde{\Omega}} \|\nabla w\|^2 \, dV 
        \;\leq\; \int_{B_{R_M}\setminus \overline{B_{R_1}}} \|\nabla w\|^2 \, dV, 
        \end{align}

       \begin{align}\label{second}
           \int_{\partial \tilde{\Omega}_{out}} w^2 \, dS
        \;\geq\; \frac{h^{n-1}(R_m)}{h^{n-1}(R_M)} 
        \int_{\partial B_{R_M}} w^2 \, dS.  \end{align}

    \begin{proof}
       In order to prove the  inequality \eqref{first}, using the 
            polar coordinates, we have
            \begin{align*}
                \int_{\tilde{\Omega}} \|\nabla w\|^2 \, dV 
                &= \int_{U_pM} \int_{R_1}^{r_u} 
                \Bigg[ \Big(\frac{\partial w}{\partial r}\Big)^2 
                + \frac{1}{h^2(r)} \|\overline{\nabla}w\|^2 \Bigg] 
                h^{\,n-1}(r)\, dr \, du,
            \end{align*}
     and, since $r_u<R_M$, we have  
            \begin{align*}
                \int_{\tilde{\Omega}} \|\nabla w\|^2 \, dV
                &\leq \int_{U_pM} \int_{R_1}^{R_M} 
                \Bigg[ \Big(\frac{\partial w}{\partial r}\Big)^2 
                + \frac{1}{h^2(r)} \|\overline{\nabla}w\|^2 \Bigg] 
                h^{\,n-1}(r)\, dr \, du \\
                &= \int_{B_{R_M}\setminus \overline{B_{R_1}}} 
                \|\nabla w\|^2 \, dV,
            \end{align*}
            which proves the claim.

  Let us prove  now \eqref{second}.  Recall that for every $q\in \partial \tilde{\Omega}_\mathrm{out},$ there esist a unique $u\in U_pM$ such that $q=exp_p(r_uu).$ We use the notation $\theta_u$ in place of $\theta(q)$ and the representation of the surface 
            measure in polar coordinates $dS=\sec (\theta_u) h^{n-1}(r_u) du$ (see \cite[p.~385]{marsden1993basic}). This gives
            \begin{align*}
                \int_{\partial \tilde{\Omega}_{out}} w^2 \, dS 
                &= \int_{U_pM} w^2 \, \sec(\theta_u)\, h^{\,n-1}(r_u)\, du.
            \end{align*}
            Since $\sec(\theta_u) \geq 1, r_u \geq R_m$, and $h(r)$ is an increasing function of $r,$ it follows that
            \begin{align*}
                \int_{\partial \tilde{\Omega}_{out}} w^2 \, dS 
                &\geq \int_{U_pM} w^2 \, h^{\,n-1}(R_m)\, du \\
                &= \frac{h^{\,n-1}(R_m)}{h^{\,n-1}(R_M)} 
                \int_{U_pM} w^2 \, h^{\,n-1}(R_M)\, du \\
                &= \frac{h^{\,n-1}(R_m)}{h^{\,n-1}(R_M)} 
                \int_{\partial B_{R_M}} w^2 \, dS.
            \end{align*}
            This concludes the proof.
      
    \end{proof}
\end{proposition}

To establish the lower bound, we need the following proposition. 
\begin{proposition}\label{lower bound inequality prop}
Let $\tilde{w}$ be a continuously differentiable function defined on $\tilde{\Omega}$.
   Then the following two inequalities hold:
    \begin{equation}
        \displaystyle \int_{B_{R_m}\setminus \overline{B_{R_1}}} \|\nabla \tilde{w}\|^2 \, dV 
        \;\leq\; \int_{\tilde{\Omega}} \|\nabla \tilde{w}\|^2 \, dV, 
    \end{equation}
 
\begin{equation}
    \displaystyle \int_{\partial B_{R_M}} \tilde{w}^2 \, dS
        \;\geq\; \frac{\sqrt{1+a} \ h^{n-1}(R_m)}{h^{n-1}(R_M)} 
        \int_{\partial \tilde{\Omega}_{out}} \tilde{w}^2 \, dS,
\end{equation}
where constant $a$ is defined in \eqref{eqn: constt a}.
\end{proposition}
The proof is same as that of Proposition~\ref{upper bound inequality}  and  Theorem~2.1 in~\cite{Basak_Chorwadwala_Verma_2025}.

\subsection{Proof of  Theorem \ref{thm:upper-lower}}
\begin{proof} [Proof of the upper bound]
Let $\psi_1$ be the first eigenfunction corresponding to $\mu_1\!\left(B_{R_M}\setminus \overline{B_{R_1}}\right)$.  
We define a function $\psi$ on $B_{R_M}\setminus\overline{B_{R_1}}$ as
\begin{align*}
   \psi := \psi_1 - \overline{\psi_1}, 
\qquad 
\overline{\psi_1} = \frac{\displaystyle \int_{\partial \tilde{\Omega}_{out}} \psi_1 \, dS}{\operatorname{Vol}(\partial \tilde{\Omega}_{out})}. 
\end{align*}
Since $\int_{\partial\tilde{\Omega}_{out}} \psi dS=0,$ function $\psi$ can be used as a test function for the variational characterization of $\mu_1(\tilde{\Omega}).$
Applying Proposition~\ref{upper bound inequality}, we obtain
\begin{align} \label{inequality for upper bound}
       \mu_1(\tilde{\Omega})\leq \frac{\displaystyle \int_{\tilde{\Omega}}\|\nabla\psi\|^2 \, dV}{\displaystyle \int_{\partial\tilde{\Omega}_{out}}\psi^2 \, dS} 
        \;\leq\; 
        \frac{h^{n-1}(R_M)}{h^{n-1}(R_m)} 
        \frac{\displaystyle \int_{B_{R_M}\setminus \overline{B_{R_1}}}\|\nabla\psi\|^2 \, dV}{\displaystyle \int_{\partial B_{R_M}}\psi^2 \, dS}.
\end{align}
Since $\overline{\psi_1}$ is constant, then
\begin{align}
    \int_{B_{R_M}\setminus \overline{B_{R_1}}} \|\nabla \psi\|^2 \, dV
    = \int_{B_{R_M}\setminus \overline{B_{R_1}}} \|\nabla \psi_1\|^2 \, dV. \label{equality 1}
\end{align}
For the denominator, we compute
\begin{align*}
    \int_{\partial B_{R_M}} \psi^2 \, dS 
    &= \int_{\partial B_{R_M}} \big(\psi_1^2 + \overline{\psi_1}^2 - 2\psi_1 \overline{\psi_1}\big) \, dS \\
    &= \int_{\partial B_{R_M}} \psi_1^2 \, dS 
       + \overline{\psi_1}^2 \, \operatorname{Vol}(\partial B_{R_M})
       - 2 \overline{\psi_1}\int_{\partial B_{R_M}} \psi_1 \, dS.
\end{align*}
By definition, $\psi_1$ is a zero-mean function on $B_{R_M}$, and therefore the last term vanishes. Hence, we get
\begin{align*}
    \int_{\partial B_{R_M}} \psi^2 \, dS 
    = \int_{\partial B_{R_M}} \psi_1^2 \, dS 
      + \overline{\psi_1}^2 \, \operatorname{Vol}(\partial B_{R_M})
    \;\geq\; \int_{\partial B_{R_M}} \psi_1^2 \, dS. 
\end{align*}
Therefore, from~\eqref{inequality for upper bound} we obtain
\begin{align*}
    \frac{\displaystyle \int_{\tilde{\Omega}} \|\nabla\psi\|^2 dV}{\displaystyle \int_{\partial\tilde{\Omega}_{out}} \psi^2 dS}
    \;\leq\;
    \frac{h^{n-1}(R_M)}{h^{n-1}(R_m)} 
    \cdot
    \frac{\displaystyle \int_{B_{R_M}\setminus \overline{B_{R_1}}} \|\nabla\psi_1\|^2 dV}{\displaystyle \int_{\partial B_{R_M}} \psi_1^2 dS}
    \;=\;
    \frac{h^{n-1}(R_M)}{h^{n-1}(R_m)} 
    \,\mu_1(B_{R_M}\setminus \overline{B_{R_1}}).
\end{align*}
Hence,
\begin{align*}
   \mu_1(\tilde{\Omega})\;\leq\; \frac{\displaystyle \int_{\tilde{\Omega}} \|\nabla\psi\|^2 dV}{\displaystyle \int_{\partial \tilde{\Omega}_{out}} \psi^2 dS}
\;\leq\;
\frac{h^{n-1}(R_M)}{h^{n-1}(R_m)} \,\mu_1(B_{R_M}\setminus \overline{B_{R_1}}). 
\end{align*}
\end{proof}

\begin{proof} [Proof of the lower bound]
The proof follows using the same strategy as that of the upper bound, except that here we take $\tilde{\psi}_1$ to be the first eigenfunction corresponding to $\mu_1(\tilde{\Omega})$ and use the test function 
$\tilde{\psi} := \tilde{\psi}_1 - \overline{\tilde{\psi}_1}$, with  
$\overline{\tilde{\psi}_1} = \dfrac{\int_{\partial B_{R_m}} \tilde{\psi}_1 \, dS}{\operatorname{Vol}(\partial B_{R_m})}$ for the variational characterization of $\mu_1(B_{R_m}\setminus \overline{B_{R_1}})$. After an application of Proposition \ref{lower bound inequality prop}, 
the rest of the argument is identical to the proof of the upper bound.

\end{proof}

\subsection{Why do we need starshapedness?}



Theorem \ref{thm:upper-lower} ensures that, whenever 
$\tilde{\Omega}_{\mathrm{out}}$ is an open, bounded, star-shaped domain and $R_1>0$, the quantity 
$\mu_1(\tilde{\Omega})$ is bounded away from zero. A natural question is whether this conclusion continues to hold for an arbitrary bounded domain in $\mathbb{R}^n$
 with Lipschitz boundary. The answer is negative, even when $\tilde{\Omega}_{\mathrm{out}}$ is assumed to be connected. This fact is illustrated by the following two-dimensional counterexample, that can be found in \cite{girouard2017spectral} in the Steklov setting in the non perforated case.
\begin{proposition}
    There exists a sequence $\{\Omega_\varepsilon\}_{\varepsilon}\subseteq \mathbb R^2$ of open, bounded sets with Lipschitz boundary, such that $\mu_1(\Omega_\varepsilon)\to 0$, as $\varepsilon\to 0^+$.
\end{proposition}
\begin{proof}
Let $\varepsilon>0$ be small enough. Let us consider  the open  rectangle centered at the origin of sides $\varepsilon$ and $\varepsilon^3$ 
\begin{equation*}
C_{\varepsilon}= \bigg(-\frac{\varepsilon}{2},\frac{\varepsilon}{2}\bigg)\times \bigg(-\frac{\varepsilon^3}{2},\frac{\varepsilon^3}{2}\bigg)
\end{equation*}
 and let us define  the two balls
\begin{equation*}
    B_{i,\varepsilon}=B_{r_i(\varepsilon)}(x_i(\varepsilon)),  \qquad i=1,2,
\end{equation*}
with radii $r_1(\varepsilon)=r_2(\varepsilon)= \sqrt{(1+\varepsilon/2)^2+\varepsilon^6/4}$, centered at the points $x_1(\varepsilon)=(-1-\varepsilon,0)$ and $x_2(\varepsilon)= (1+\varepsilon,0)$, respectively.
Let us consider the following sequence of perforated dumbells $\{\Omega_\varepsilon\}\subset \mathbb{R}^2 $, defined as follows
\begin{equation*}
	\Omega_\varepsilon = B_{1,\varepsilon}\cup C_{\varepsilon}\cup (B_{2,\varepsilon}\setminus \overline{B_{R_1}(x_2(\varepsilon))}),
\end{equation*}
where $B_{R_1}(x_2(\varepsilon))$ is the concentric ball contained in $B_{r_2(\varepsilon)}(x_2(\varepsilon))$ of radius $0<R_1<1$. Let us consider the following continuous function
\begin{equation*}
	v(x,y)= \begin{cases}
	\sin \bigg(\displaystyle\frac{2\pi x}{\varepsilon}\bigg) & \text{in} \, C_{\varepsilon}\\
	0 & \text{elsewhere}.
	\end{cases}
\end{equation*}
Let us stress that $v(x,y)$ is a good test function for the first non zero Steklov--Neumann eigenvalue, since
\begin{equation*}
    \int_{\partial \Omega_\varepsilon\setminus \partial B_{R_1}(x_2(\varepsilon))}v\,d\mathcal{H}^{n-1}  = \int_{\partial C_\varepsilon}v\,d\mathcal{H}^{n-1} = 2\int_{-\frac{\varepsilon}{2}}^{\frac{\varepsilon}{2}}\sin\bigg(\displaystyle\frac{2\pi x}{\varepsilon}\bigg)\,dx = -\frac{\varepsilon}{\pi}\cos \bigg(\displaystyle\frac{2\pi x}{\varepsilon}\bigg)\Bigg|_{-\frac{\varepsilon}{2}}^{\frac{\varepsilon}{2}}=0.
\end{equation*}
The denominator in the Rayleigh quotient of $\mu_1(\Omega_\varepsilon)$ becomes

\begin{equation*}
\begin{split}
\int_{\partial \Omega_\varepsilon}v^2 \,d\mathcal{H}^{n-1}&= \int_{\partial C_\varepsilon}v^2\,d\mathcal{H}^{n-1} = 2\int_{-\frac{\varepsilon}{2}}^{\frac{\varepsilon}{2}}\sin^2 \bigg(\displaystyle\frac{2\pi x}{\varepsilon}\bigg)\,dx\\
&=4\int_0^{\frac{\varepsilon}{2}} \sin^2\bigg(\displaystyle\frac{4\pi x}{\varepsilon}\bigg)\,dx
=4\displaystyle\int_0^{\frac{\varepsilon}{2}}\frac{1-\cos\bigg(\displaystyle\frac{4\pi x}{\varepsilon}\bigg)}{2}\,dx=\varepsilon.
\end{split}
\end{equation*}
Moreover, since
\begin{equation*}
	|\nabla v|^2 = \bigg(\frac{\partial v}{\partial x}\bigg)^2 = \bigg(\frac{2\pi}{\varepsilon}\bigg)^2 \cos^2\bigg(\displaystyle\frac{2\pi x}{\varepsilon}\bigg),
\end{equation*}
The numerator is
\begin{equation*}
\begin{split}
\int_{\Omega_\varepsilon}|\nabla v|^2 \,dx\,dy&= \int_{C_\varepsilon}|\nabla v|^2\,dx\,dy = \bigg(\frac{2\pi }{\varepsilon}\bigg)^2\int_{-\frac{\varepsilon^3}{2}}^{\frac{\varepsilon^3}{2}}\int_{-\frac{\varepsilon}{2}}^{\frac{\varepsilon}{2}} \cos^2\bigg(\displaystyle\frac{4\pi x}{\varepsilon}\bigg)\,dx\,dy \\
&= 2\bigg(\frac{2\pi }{\varepsilon}\bigg)^2\varepsilon^3\int_0^{\frac{\varepsilon}{2}}\frac{1+\cos\bigg(\displaystyle\frac{4\pi x}{\varepsilon}\bigg)}{2}\,dx=2\pi^2\varepsilon^2.
\end{split}
\end{equation*}
In this way, since $v$ is zero on $\partial B_{R_1}(x_2)$, we get
\begin{equation*}
	\mu_1(\Omega_\varepsilon)\le \frac{\displaystyle \int_{\Omega_\varepsilon}|\nabla u|^2\,dx}{\displaystyle \int_{\partial \Omega_\varepsilon}u^2 \,d\mathcal{H}^{n-1}}= 2\pi^2\varepsilon,
\end{equation*}
and eventually
\begin{equation*}
	\mu_1(\Omega_\varepsilon)\to 0, \qquad \qquad \text{as}\;\;\varepsilon\to 0. 
\end{equation*}
    \end{proof}

The previous counterexample provides tell us that the first non zero Steklov--Neumann eigenvalue can become arbitrarily small when the domain is not star-shaped. 

\section{Asymptotic behavior with respect to the radius of the hole}
In this section, we study the behavior of the first non zero Steklov--Neumann eigenvalue and the corresponding eigenfunctions on a bounded domain with a spherical hole, as the radius of the spherical hole approaches to 0. This is the content of the Theorem \ref{thm:convergence}. 


 We denote by $\Omega_r=\Omega_{\mathrm{out}} \setminus\overline{B_r}$, where $\Omega_{\mathrm{out}} \subset\mathbb{R}^n$ is an open, bounded and Lipschitz set and $B_r$ is the ball centered at the origin of radius $r>0$ such that $\overline{B_r}\subset \Omega_{\mathrm{out}}$.
We also denote the concentric annulus in the following way $A_{r,R}=B_R\setminus\overline{B_r}$, with $R>r>0$.
 
In order to prove Theorem \ref{thm:convergence}, we need to consider a particular sequence of eigenfunctions $\{u_1^r\}_r$ corresponding to $\mu_1(\Omega_r)$. Differently from the Steklov-Dirichlet case, in which it sufficed to extend the sequence of eigenfunctions to zero inside the hole, in this case the situation is a bit more delicate, since we need to extend them by harmonicity inside $B_r$. More precisely,  we need control the rate of convergence of the gradient of the harmonic extension $h$ in $L^2(B_r)$  as the radius $r$ approaches zero and use the following technical Lemma, that is an application of \cite[Lemma 12]{girouard2021steklov}. 
\begin{lemma}\label{lem:estimateharmonicgradient}
    For any $0<r<R$, let $u\in C^\infty(\overline{A_{r,R}})$ be such that
    \begin{equation*}
        \begin{cases}
         \Delta u = 0 & \text{in}\;\; A_{r,R},\\
         \partial_\nu u = 0 & \text{in}\;\; \partial B_r,
        \end{cases}
    \end{equation*}
    and let us consider the function $h$, which is the harmonic extension of $u$ in $B_r$, that is
    \begin{equation*}
        \begin{cases}
        \Delta h = 0 &\text{in}\;\; B_r,\\
        h=u &\text{on}\;\; \partial B_r.
        \end{cases}
    \end{equation*}
  Then, as $r/R\to 0$, we have
  \begin{align} \label{eq:gradestimate}
      \int_{B_r}|{\nabla h}|^2\,dx\leq 5 \bigg(\frac{r}{R}\bigg)^n\bigg(1+O\bigg(\frac{r}{R}\bigg)^n\bigg)\int_{A_{r,R}}|{\nabla u}|^2\,dx. 
  \end{align}
\end{lemma}
  Now we present a proof of the main result of this section.
\begin{proof}[Proof of Theorem \ref{thm:convergence}] In the proof, we will use the following notation: let $\varepsilon>0$ be small enough and let us consider $\Omega_{r_\varepsilon}= \Omega_\mathrm{out} \setminus \overline{B_{r_\varepsilon}}$, where $B_{r_\varepsilon}$ is a ball centered at the origin with radius $ r_\varepsilon$, with $0<r_\varepsilon<\varepsilon$, such that $\overline{B}_{r_\varepsilon}\subset \Omega_\mathrm{out}$, and  $r_\varepsilon/\varepsilon\to 0^+$ as $\varepsilon \to 0^+$. From now onwords, we will consider $u_1^{r_\varepsilon}$ as an eigenfunction corresponding to the first non zero Steklov--Neumann eigenvalue $\mu_1(\Omega_{r_\varepsilon})$, such that $\|u_1^{r_\varepsilon}\|_{L^2(\partial \Omega_{out})}=1$.

\textbf{Step 1. (An upper bound for $\mu_1(\Omega_{r_\varepsilon})$)} In this step we will prove that
\begin{equation}\label{eq:uppbound}
    \mu_1(\Omega_{r_\varepsilon})\le \sigma_1(\Omega_{out}).
\end{equation}
   This follows immediately by choosing as a test function in the variational characterization \eqref{characterization} of $\mu_1(\Omega_{r_\varepsilon})$, an eigenfunction corresponding to the first non-trivial Steklov eigenvalue $\sigma_1(\Omega_{out})$, let us say $\tilde{u}_1$. Then, $\tilde{u}_1$ is orthogonal to the constant function on $\partial \Omega_{\mathrm{out}},$ that is
   \begin{align*}
       \int_{\partial \Omega_{\mathrm{out}}} \tilde{u}_1 \, dS =0.
   \end{align*}
    Therefore, we get 
    \begin{align*}
        \mu_1(\Omega_{r_{\varepsilon}})\leq \frac{\displaystyle \int_{\Omega_{r_{\varepsilon}}} |\nabla \tilde{u}_1|^2 \, dV}{\displaystyle\int_{\partial \Omega_{\mathrm{out}}}\tilde{u}_1^2 \,dS}\le\frac{\displaystyle\int_{\Omega_{\mathrm{out}}} |\nabla \tilde{u}_1|^2 \, dV}{\displaystyle\int_{\partial\Omega_{\mathrm{out}}}\tilde{u}_1^2\, dS} =\sigma_1(\Omega_{\mathrm{out}}),
    \end{align*}
\noindent for every $\varepsilon$ small enough. Then, applying Brock inequality (\cite{Bro2001}), we get 
    \begin{align*}
        \mu_1(\Omega_{r_{\varepsilon}})\leq \sigma_1(\Omega_\mathrm{out})\leq \sigma_1(B_R)= \frac{1}{R},
    \end{align*}
    for every $\varepsilon$ small enough, where $B_R$ is the ball having same measure as $\Omega_{\mathrm{out}}.$
     
  \textbf{Step 2. (Convergence of the eigenvalues)}  In the previous step, we have proved that the sequence of strictly positive eigenvalues $ \{\mu_1 (\Omega_{r_{\varepsilon}})\}_{\varepsilon}$ is uniformly bounded for small $\varepsilon >0$. Then, there exist a subsequence, still denoted by $\{\mu_1 (\Omega_{r_{\varepsilon}})\}_{\varepsilon}$, such that
     \begin{align*}
         \lim_{\varepsilon \to 0} \inf \mu_1(\Omega_{r_{\varepsilon}}) = \tilde{\sigma}.
     \end{align*}
Let us now extend by harmonicity $u_1^{r_\varepsilon}$ inside $B_{r_\varepsilon}$, i.e. let us consider
the function $U_1^{r_\varepsilon}\in H^1(\Omega_\mathrm{out})$, defined as
\begin{equation*}
    \begin{cases}
        U_1^{r_\varepsilon}=u_1^{r_\varepsilon} & \text{in}\;\; \overline{\Omega_{r_\varepsilon}}\\
        \Delta U_1^{r_\varepsilon}=0 & \text{in}\;\; B_{r_\varepsilon}.
    \end{cases}
\end{equation*}
We want to prove that $U_1^{r_\varepsilon}$ is equibounded in $L^2(\Omega_{out})$. If we consider the following eigenvalue problem with Robin boundary condition and Robin boundary  parameter $\beta=1$,
\begin{equation*}
\begin{cases}
    \Delta v = \lambda v & \text{in}\;\; \Omega_{out}\\
        \partial_\nu v+v=0 & \text{in}\;\; \partial \Omega_{out},
\end{cases}
\end{equation*}
Then, the first Robin eigenvalue has the following variational characterization
\begin{equation*}
  \lambda_1(\Omega_{out}) = \inf_{0\neq v \in H^1(\Omega_{out})}\frac{\displaystyle\int_{\Omega_{out}}|{\nabla v}|^2\,dx + \int_{\partial \Omega_{out}}v^2\,d\mathcal{H}^{n-1}}{\displaystyle\int_{\Omega_{out}}v^2\,dx}.  
\end{equation*}
Since $U_1^{r_\varepsilon}\in H^1(\Omega_{out})$, we have
\begin{equation*}
    \|U_1^{r_\varepsilon}\|_{L^2(\Omega_{out})}^2\le \lambda_1(\Omega_{out})^{-1}(\|\nabla U_1^{r_\varepsilon}\|_{L^2(\Omega_{out})}^2+\|U_1^{r_\varepsilon}\|_{L^2(\partial \Omega_{out})}^2).
\end{equation*}
By the normalization condition, we know that
\begin{equation*}
    \|\nabla U_1^{r_\varepsilon}\|_{L^2(\Omega_{out})}^2= \|\nabla U_1^{r_\varepsilon}\|_{L^2(B_{r_\varepsilon})}^2+\|\nabla u_1^{r_\varepsilon}\|_{L^2(\Omega_{r_\varepsilon})}^2=\|\nabla U_1^{r_\varepsilon}\|_{L^2(B_{r_\varepsilon})}^2+\mu_1(\Omega_{r_\varepsilon}).
\end{equation*}
Therefore
\begin{equation*}
     \|U_1^{r_\varepsilon}\|^2_{L^2(\Omega_{\mathrm{out}})}\le \lambda_1(\Omega_{out})^{-1}(\|\nabla U_1^{r_\varepsilon}\|_{L^2(B_{r_\varepsilon})}^2+\mu_1(\Omega_{r_\varepsilon})+1).
\end{equation*}
Let us estimate now $\|\nabla U_1^{r_\varepsilon}\|_{L^2(B_{r_\varepsilon})}$. By applying Lemma \ref{lem:estimateharmonicgradient} in the annulus $A_{r_\varepsilon,\varepsilon}$, which is well contained in $\Omega_{r_\varepsilon}$ for $\varepsilon$ small enough, we get
\begin{equation*}
\begin{split}
    \| \nabla U_1^{r_\varepsilon}\|^2_{L^2(B_{r_\varepsilon})}&\le 5\bigg(\frac{r_\varepsilon}{\varepsilon}\bigg)^n\bigg(1+O\bigg(\frac{r_\varepsilon}{\varepsilon}\bigg)^n\bigg)\|\nabla u_1^{r_\varepsilon}\|^2_{L^2(A_{r_\varepsilon,\varepsilon})}  \\
    & \le C\|\nabla u_1^{r_\varepsilon}\|^2_{L^2(A_{r_\varepsilon,\varepsilon})}\le C \|\nabla u_1^{r_\varepsilon}\|^2_{L^2(\Omega_{r_\varepsilon})}= C\mu_1(\Omega_{r_\varepsilon}),
\end{split}
\end{equation*}
where we have used again that $A_{r_\varepsilon,\varepsilon}\subseteq \Omega_{r_\varepsilon}$. This proves that
\begin{equation*}
      \|U_1^{r_\varepsilon}\|^2_{L^2(\Omega_\mathrm{out})}\le \lambda_1(\Omega_{out})^{-1}[(C+1)\mu_1(\Omega_{r_\varepsilon})+1].
\end{equation*}
By the step one, $\mu_1(\Omega_{r_\varepsilon})\le \sigma_1(\Omega_{\mathrm{out}})$ for $\varepsilon$ small enough. This implies that $U_1^{r_\varepsilon}$ is equibounded in $L^2(\Omega_{\mathrm{out}})$ and, therefore, there exists a subsequence still denoted by  $\{U_1^{r_\varepsilon}\}_\varepsilon$ converges strongly in $L^2(\Omega_{\mathrm{out}})$ to some $\overline{u_1} \in H^1(\Omega_{\mathrm{out}})$ and such that $\nabla U_1^{r_\varepsilon} \rightharpoonup \nabla \overline{u_1}$ weakly in $L^2(\Omega_{\mathrm{out}}).$ Moreover, by compactness of the trace operator we have that   $\{U_1^{r_\varepsilon}\}_\varepsilon$   converges strongly in $L^2(\partial\Omega_{\mathrm{out}})$ almost everywhere on  $\partial\Omega_{\mathrm{out}}.$

Now we use a test function $ \phi\in H^1(\Omega_{\mathrm{out}}),$ in the weak formulation of the Steklov-Neumann eigenvalue problem \eqref{weak formulation SN}. 
 Since $U_1^{r_\varepsilon}$ is an eigenfunction corresponding to the eigenvalue $\mu_1(\Omega_{r_\varepsilon}).$ Then,
 \begin{align}\label{passing limit}
     \int_{\Omega_{r_\varepsilon}} \nabla U_1^{r_\varepsilon}\,\, \nabla\phi \, dV= \mu_1(\Omega_{r_\varepsilon}) \int_{\partial \Omega_{\mathrm{out}}} U_1^{r_\varepsilon} \phi \, dS.
 \end{align}
Since $\nabla U_1^{r_\varepsilon}$ converges weakly in $L^2(\Omega_{\mathrm{out}})$  to $\nabla \overline{u_1},$ then

\begin{align*}
    \lim_{\varepsilon \to 0} \int_{\Omega_{r_\varepsilon}}  \nabla U_1^{r_\varepsilon}\,\, \nabla\phi \, dV = \int_{ \Omega_{\mathrm{out}}} \nabla \overline{u_1} \, \nabla \phi \,dV ,
\end{align*}
and since $U_1^{r_\varepsilon}$ converges strongly in $L^2(\partial \Omega_{\mathrm{out}})$, then
\begin{align*}
    \lim_{\varepsilon \to 0} \int_{\partial \Omega_{\mathrm{out}}} U_1^{r_\varepsilon} \phi \, dS = \int_{\partial \Omega_{\mathrm{out}}} \overline{u_1} \phi \, dS.
\end{align*}
Now, passing to the lim inf in  \eqref{passing limit} as $\varepsilon \to 0,$ we get 
 \begin{align} \label{inf limit}
      \int_{ \Omega_{\mathrm{out}}} \nabla \overline{u_1} \, \nabla \phi \,dV =\tilde{\sigma} \int_{\partial \Omega_{\mathrm{out}}} \overline{u_1} \phi \, dS, \quad \text{for all} \quad \phi \in H^1(\Omega_{\mathrm{out}}).
 \end{align}
This means that $\tilde{\sigma}$ belongs to the spectrum of the Dirichlet-to-Neumann operator.
If $\tilde{\sigma}=0,$ then from \eqref{inf limit} we would have $\overline{u_1}$  constant, and the orthogonality condition 
\begin{align*}
    0=\int_{\partial \Omega_\mathrm{out}} U_1^{r_\varepsilon} \to \int_{\partial \Omega_{\mathrm{out}}} \overline{u_1}, \quad\text{as}\quad r_\varepsilon\to 0,
\end{align*}
would imply that $\overline{u_1}=0$ on ${\partial \Omega_{\mathrm{out}}}$. But this is not possible since the normalization condition holds. Therefore, $\tilde{\sigma} >0$ and therefore $\sigma_1({\Omega_{\mathrm{out}}})\leq \tilde{\sigma}.$    
 Then, 
 \begin{align*}
     \sigma_1({\Omega_{\mathrm{out}}})\leq \tilde{\sigma}= \liminf_{\varepsilon \to 0} \mu_1(\Omega_{r_\varepsilon})\leq\limsup_{\varepsilon \to 0} \mu_1(\Omega_{r_\varepsilon})\leq \sigma_1({\Omega_{\mathrm{out}}}).
 \end{align*}
 Hence, we conclude that 
 \begin{align*}
     \lim_{\varepsilon \to 0} \mu_1(\Omega_{r_\varepsilon})= \sigma_1({\Omega_{\mathrm{out}}}).
 \end{align*}

 \textbf{Step 3. (Convergence of eigenfunction)}  
We now recall the convergence properties of the sequence 
$\{U_1^{r_\varepsilon}\}$.  
 We have
\begin{align*}
   U_1^{r_\varepsilon} \longrightarrow \overline{u_1}
    \qquad \text{strongly in } \quad L^{2}(\Omega_{\mathrm{out}}), 
\end{align*}
and 
\begin{align*}
\nabla U_1^{r_\varepsilon} \rightharpoonup 
\nabla \overline{u_1}
    \qquad \text{weakly in } L^{2}(\Omega_{\mathrm{out}}).
\end{align*}
Moreover, we obtained the strong convergence of the traces:
\begin{align*}
  U_1^{r_\varepsilon}\big|_{\partial\Omega_{\mathrm{out}}}
\longrightarrow \overline{u_1}\big|_{\partial\Omega_{\mathrm{out}}}
    \qquad \text{in } L^{2}(\partial\Omega_{\mathrm{out}}).  
\end{align*}

To prove the strong convergence of $ U_1^{r_\varepsilon}$ to $\overline{u_1}$ in $H^1( \Omega_{\mathrm{out}})$, we only need to prove the strong convergence of  $ \nabla U_1^{r_\varepsilon}$ to $\nabla \overline{u_1}$ in $L^2(\Omega_{\mathrm{out}}).$ We consider $U_1^{r_\varepsilon}$ as a test function  in the weak formulation of the classical Steklov eigenvalue problem on $\Omega_\mathrm{out}$, arriving to 
\begin{align}
      \int_{ \Omega_{\mathrm{out}}} \nabla \overline{u_1} \, \nabla U_1^{r_\varepsilon} \,dV =\sigma_1(\Omega_{\mathrm{out}}) \int_{\partial \Omega_{\mathrm{out}}} \overline{u_1}\, U_1^{r_\varepsilon} \, dS.
\end{align}
Then,
\begin{align*}
    \|\nabla U_1^{r_\varepsilon}-\nabla \overline{u_1}\|_{L^2(\Omega_{\mathrm{out}})}
    &=\|\nabla U_1^{r_\varepsilon}\|^2_{L^2( \Omega_{\mathrm{out}})} + \|\nabla \overline{u_1}\|^2_{L^2( \Omega_{\mathrm{out}})} -2 \int_{ \Omega_{\mathrm{out}}} \nabla \overline{u_1} \, \nabla U_1^{r_\varepsilon} \, dV\\
    &= \|\nabla U_1^{r_\varepsilon}\|^2_{L^2( B_{r_\varepsilon})} +\mu_1( \Omega_{r_\varepsilon}) + \sigma_1( \Omega_{\mathrm{out}}) -2 \sigma_1(\Omega_{\mathrm{out}})\int_{ \partial \Omega_{\mathrm{out}}}  \overline{u_1} \,  U_1^{r_\varepsilon} \, dS.\\
\end{align*}
Applying Lemma \ref{lem:estimateharmonicgradient} and considering that $A_{r_\varepsilon,\varepsilon}\subseteq \Omega_{r_\varepsilon}$, we get 
\begin{equation*}
\begin{split}
    \| \nabla U_1^{r_\varepsilon}\|_{L^2(B_{r_\varepsilon})}&\le 5\bigg(\frac{r_\varepsilon}{\varepsilon}\bigg)^n\bigg(1+O\bigg(\frac{r_\varepsilon}{\varepsilon}\bigg)^n\bigg)\|\nabla u_1^{r_\varepsilon}\|_{L^2(A_{r_\varepsilon,\varepsilon})}  \\
    & \le 5\bigg(\frac{r_\varepsilon}{\varepsilon}\bigg)^n\bigg(1+O\bigg(\frac{r_\varepsilon}{\varepsilon}\bigg)^n\bigg)\|\nabla u_1^{r_\varepsilon}\|_{L^2(\Omega_{r_\varepsilon})} \\
     &=5\bigg(\frac{r_\varepsilon}{\varepsilon}\bigg)^n\bigg(1+O\bigg(\frac{r_\varepsilon}{\varepsilon}\bigg)^n\bigg)\mu_1(\Omega_{r_\varepsilon}),
\end{split}
\end{equation*}
which implies
\begin{equation*}
    \lim_{\frac{r_\varepsilon}{\varepsilon}\to 0^+}\| \nabla U_1^{r_\varepsilon}\|_{L^2(B_{r_\varepsilon})}= 0.
\end{equation*}
Furthermore,  $$\mu_1( \Omega_{r_\varepsilon}) \to \sigma_1( \Omega_{\mathrm{out}}), \quad\text{as} \quad \varepsilon\to 0, $$
and the normalization together with the strong convergence of the traces on $\partial \Omega_{\mathrm{out}}$ yields
$$ \int_{ \partial \Omega_{\mathrm{out}}}  \overline{u_1} \,  U_1^{r_\varepsilon} \, dS \to  \int_{ \partial \Omega_{\mathrm{out}}}  \overline{u_1} ^2\, dS=1.$$ 
Hence
\begin{align*}
     \|\nabla U_1^{r_\varepsilon}-\nabla \overline{u_1}\|_{L^2(\Omega_{\mathrm{out}})} \to 0 \quad \text{as} \quad \varepsilon \to 0.
\end{align*}
\end{proof}

The asymptotic convergence established in Theorem~\ref{thm:convergence}
naturally leads to a family of isoperimetric inequalities for domains with
small holes, which we state below.

\begin{cor}\label{thm:isoperimetric1}
Let $\Omega_r = \Omega_{\mathrm{out}} \setminus \overline{B_r}$, where
$\Omega_{\mathrm{out}} \subset \mathbb{R}^n$ is an open, bounded, and
connected set with Lipschitz boundary, and $B_r$ is a ball of radius $r>0$
centered at the origin such that $\overline{B_r} \subset \Omega_{\mathrm{out}}$.
Then, there exists $r_2^M = r_2^M(\Omega_{\mathrm{out}})$, depending only on
the geometry of $\Omega_{\mathrm{out}}$, such that for every
$r \in (0, r_2^M)$,
\begin{equation}\label{eq:isopsecondSD}
\mu_1(\Omega_r) \le \mu_1(A_{r,R_M}),
\end{equation}
where $R_M > 0$ is the radius of the ball having the same measure as
$\Omega_{\mathrm{out}}$. The equality case holds if and only if $\Omega_r = A_{r,R_M}$.
\end{cor}

The next result is obtained by replacing the measure constraint with a
perimeter constraint and applying the Weinstock inequality (\cite{BFNT2021}).

\begin{cor}\label{thm:isoperimetric2}
Let $\Omega_r = \Omega_{\mathrm{out}} \setminus \overline{B_r}$, where
$\Omega_{\mathrm{out}} \subset \mathbb{R}^n$ is an open, bounded, and convex
set, and $B_r$ is a ball of radius $r>0$ centered at the origin such that
$\overline{B_r} \subset \Omega_{\mathrm{out}}$.
Then there exists $r_2^P = r_2^P(\Omega_{\mathrm{out}})$, depending only on
the geometry of $\Omega_{\mathrm{out}}$, such that for every
$r \in (0, r_2^P)$,
\begin{equation}\label{eq:isopsecondSDP}
\mu_1(\Omega_r) \le \mu_1(A_{r,R_P}),
\end{equation}
where $R_P > 0$ denotes the radius of the ball having the same perimeter as
$\Omega_{\mathrm{out}}$. The equality case holds if and only if $\Omega_r = A_{r,R_P}$.
\end{cor}

We will give only the proof of Corollary~\ref{thm:isoperimetric2} and we stress that it is just a simple appilication of the Weinstock inequality and limiting arguments that follow from Theorem \ref{thm:convergence}. The proof of Corollary~\ref{thm:isoperimetric1} follows the same arguments, where the Weinstock inequality is replaced by the Brock inequality.

\begin{proof}[Proof of Corollary~\ref{thm:isoperimetric2}]
By the Weinstock inequality, since $\Omega_{\mathrm{out}}$ is convex and has the same perimeter as $B_{R_P}$, we have
\[
\sigma_1(\Omega_{\mathrm{out}}) \le \sigma_1(B_{R_P}),
\]
and the equality holds if and only if $\Omega_{\mathrm{out}} = B_{R_P}$.\\
If $\Omega_{\mathrm{out}} = B_{R_P}$, then $\Omega_r = A_{r,R_P}$ and equality holds trivially. Therefore we can assume that $\Omega_{\mathrm{out}} \neq B_{R_P}$. Then the Weinstock inequality is strict and there exists $\delta > 0$ such that
\[
\sigma_1(\Omega_{\mathrm{out}}) \le \sigma_1(B_{R_P}) - 2\delta.
\]
By Theorem \ref{thm:convergence} we have that
\[
\mu_1(\Omega_r) \to \sigma_1(\Omega_{\mathrm{out}}) \quad \text{as } r \to 0^+.
\]
Hence there exists $r_1=r_1(\Omega_{\mathrm{out}}) > 0$ such that for every $r < r_1$,
\[
\mu_1(\Omega_r) \le \sigma_1(\Omega_{\mathrm{out}}) + \delta.
\]
Therefore,
\[
\mu_1(\Omega_r) \le \sigma_1(B_{R_P}) - \delta.
\]
Always by Theorem \ref{thm:convergence}, we have that
\[
\mu_1(A_{r,R_P}) \to \sigma_1(B_{R_P}) \quad \text{as } r \to 0^+,
\]
and so there exists $r_2=r_2(\Omega_\mathrm{out}) > 0$ such that for every $r < r_2$,
\[
\sigma_1(B_{R_P}) - \delta\le \mu_1(A_{r,R_P})  .
\]
Hence, for every $r < \min\{r_1,r_2\}=r^P_2$, we obtain
\[
\mu_1(\Omega_r) \le \sigma_1(B_{R_P}) - \delta 
\le \mu_1(A_{r,R_P}),
\]
which proves the desired inequality.
Conversely, if equality holds, then equality must hold in the Weinstock inequality, which implies $\Omega_{\mathrm{out}} = B_{R_P}$.
\end{proof}

\begin{remark}
In dimension two, Corollary~\ref{thm:isoperimetric2} remains valid within
the class of planar simply connected domains (see \cite{Wein1954}).
\end{remark}

\section{About the nodal domains}\label{sec:6}
Let $\Omega= \Omega_\mathrm{out}\setminus\overline{B_r}$  and let us denote by $V_k$ the eigenspace corresponding to $\mu_k(\Omega)$. Before stating the result, we introduce the notion of a nodal domain,
\begin{definition}
    We will call $\Omega_j$ a nodal domain of $u\in V_k$ in $\Omega$, a connected component of the set $\{x\in \Omega: u(x)\neq 0\}$. In particular, we will denote by
    \begin{equation}
        \Delta_k = \max_{v \in V_k} \sharp \{\text{nodal domains of v} \}.
    \end{equation}
\end{definition}

In this section, we study the number of nodal domains of the eigenfunctions corresponding to the first nonzero Steklov--Neumann eigenvalue. 
In the radial case (see section \ref{subsec:radial}), we recall the explicit expression of the eigenfunctions given by
\begin{equation*}
    u^i_n(x)= \left(r+\frac{R_1^n}{(n-1)r^{n-1}}\right)\frac{x_i}{r}.
 \end{equation*}
It is immediately evident that for every $i=1,...,n$, the eigenfunctions have exactly two nodal domains and that the nodal line is the set where $x_i=0$. This observation naturally raises the question of whether the same
property holds for a general doubly connected domain $\Omega$.
The answer to this question is affermative and it is given in the proof of Theorem \ref{prop:nodaldomains}.\\

 The arguments of the proof follow the idea contained in \cite{AM1994, KS1969}, with  modifications due to the  presence of a hole in $\Omega_\mathrm{out}$.

\begin{proof} [Proof of Theorem \ref{prop:nodaldomains}]

    By contradiction, let us assume that there exists $ u \in V_{1}$ with $\Delta$ nodal domains $\Omega_1,\dots,\Omega_\Delta$, such that
    \begin{equation*}
        \Delta \ge 2+m_0=3,
    \end{equation*}
    where $m_0=1$ is the multiplicity of the trivial eigenvalue $\mu_0(\Omega)$.
    Let $u_0=1$ the basis of the eigenspace $W_0$ and let us consider the following function
    \begin{equation*}
        v = 
        \begin{cases}
            \displaystyle \sum_{l=1}^{\Delta -1}\alpha_l \cdot u \chi_{\Omega_l} & \text{in} \,\, \Omega\setminus \overline{\Omega}_\Delta\\
            \displaystyle 0 & $\text{in}$\,\, \Omega_\Delta,
        \end{cases}
    \end{equation*}
    where $\chi_{\Omega_l}(\cdot)$ is the characteristic function of the nodal domain $\Omega_l$ and $\alpha_l \in \mathbb{R}$ not all zero, for all $l=0,\dots, \Delta-1$. In particular, we can choose $\alpha_1,\dots,\alpha_{\Delta-1}$ in such a way that
    \begin{equation*}
        \int_{\partial \Omega_0} v u_0\,d\mathcal{H}^{n-1}=0.
    \end{equation*}
     This makes $v$ an admissible function for $\mu_1$ and hence, for every $l=0,\dots, \Delta -1$, an integration by parts gives
    \begin{equation*}
    \begin{split}
        \int_{\Omega_l} |\nabla v|^2\,dx &= \int_{\partial \Omega_l} \frac{\partial v}{\partial \nu}v \,d\mathcal{H}^{n-1}\\
        &=\alpha_l^2 \int_{\partial \Omega_l\cap \partial \Omega_0} \frac{\partial u}{\partial \nu}u \,d\mathcal{H}^{n-1} = \alpha_l^2 \mu_1  \int_{\partial \Omega_l\cap \partial \Omega_0} u^2 \,d\mathcal{H}^{n-1}.
    \end{split}
    \end{equation*}
    Summing up to $\Delta -1$, we get
    \begin{equation*}
        \int_{\Omega}|\nabla v|^2 = \mu_1\int_{\partial \Omega_0}v^2\,d\mathcal{H}^{n-1}.
    \end{equation*}
    Hence, $v$ is an eigenfunction associated to $\mu_1$. Since $v\equiv 0$ in $\Omega_\Delta$ and it is harmonic, it must be $v\equiv 0$ in $\Omega$, which is a contradiction. Therefore, we must have
     \begin{equation} \label{nodal ineq}
        \Delta_{1}\le 1+ m_0=2.
    \end{equation}
Next, let $v$ be an eigenfunction associated to $\mu_1 (\Omega)$. Since it holds 
    $$ \int_{\partial \Omega_0}v \;d\mathcal{H}^{n-1}=0,$$
    and by the regularity of $v$ ($v$ is a harmonic function, hence $C^\infty(\Omega)\cap C(\Bar{\Omega})$), we have that 
     $v$ is sign changing, which means that $\Delta_2 \ge 2$. Then applying the inequality \eqref{nodal ineq}, we have \begin{equation*} \Delta_2 = 2.
    \end{equation*}
\end{proof}



\begin{remark}
A notable difference between the Steklov--Dirichlet and the
Steklov--Neumann problems concern the structure of the nodal domains of
the associated eigenfunctions. In the Steklov--Dirichlet case, the
Dirichlet condition imposed on the inner boundary forces the eigenfunctions
to vanish there, which strongly influences both their sign distribution
and the geometry of their nodal sets. As a consequence, the nodal lines
tend to be anchored to the Dirichlet boundary, and their behavior is
closely tied to the topology of the perforation (see \cite[Proposition 2.5]{sannipoli2025estimates}).

By contrast, in the Steklov--Neumann problem, the Neumann condition on the
inner boundary allows eigenfunctions to attain non-zero values on
$\Gamma_1$. This additional flexibility leads to a markedly different nodal
structure, as nodal sets are no longer constrained to intersect the inner
boundary. The analysis of nodal domains therefore requires distinct
arguments, and classical techniques developed for the Steklov--Dirichlet
problem cannot be applied directly.
\end{remark}

\section*{Acknowledgements}
This work has been partially supported by GNAMPA group of INdAM. \\R. Sannipoli was supported by the grant no. 26-21940S
of the Czech Science Foundation. S. Basak is supported by the University Grants Commission, India. S. Verma acknowledges the project grant provided by SERB-SRG sanction order No. SRG/2022/002196.
G. Paoli was supported by "INdAM - GNAMPA Project", codice CUP E5324001950001
    and  by the Project PRIN 2022 PNRR:  "A sustainable and trusted Transfer Learning platform for Edge Intelligence (STRUDEL)", CUP E53D23016390001, in the framework of European Union - Next Generation EU program

\section*{Conflicts of interest and data availability statement}
The authors declare that there is no conflict of interest. Data sharing not applicable to this article as no datasets were generated or analyzed during the current study.

\Addresses
\bibliographystyle{plain}
\bibliography{Ref}

\end{document}